\documentclass[reqno,12pt]{article}

\usepackage{verbatim,graphicx}

\RequirePackage[latin1]{inputenc} \RequirePackage[english]{babel}
\RequirePackage{amssymb} \RequirePackage{amsfonts}
\RequirePackage{amsmath} \RequirePackage{amsthm}
\RequirePackage{latexsym}

\newcommand{\tbf}[1]{{\bf{#1}}}

\newcommand{\R}{\mathbb{R}}

\newcommand{\N}{\mathbb{N}}

\newcommand{\Var}{\operatorname{Var}}

\newcommand{\argmin}{\operatorname{argmin}}

\theoremstyle{plain}

\newtheorem{lema}{Lemma}

\newtheorem{teor}{Theorem}

\newtheorem{teor*}{Teorema}
\theoremstyle{definition}
\newtheorem{rem}{Remark}

\begin{document}
\thispagestyle{empty} \vspace*{1cm}
\begin{center}
{\bf\large A note on kernel density estimation at a parametric rate}
 \end{center}

\begin{center}
{\sc  J. E. Chac\'{o}n, J. Montanero, A. G. Nogales}\\
\end{center}

\begin{center}
Dpto. de Matem\'{a}ticas, Universidad de Extremadura, SPAIN.
\end{center}
\vspace{.4cm}
\begin{quote}
\hspace{\parindent} {\small {\sc Abstract.} In the context of kernel
density estimation, we give a characterization of the kernels for
which the parametric mean integrated squared error rate $n^{-1}$ may
be obtained, where $n$ is the sample size. Also, for the cases where
this rate is attainable, we give an asymptotic bandwidth choice that
makes the kernel estimator consistent in mean integrated squared
error at that rate and a numerical example showing the superior
performance of the superkernel estimator when the bandwidth is
properly chosen.

}
 \end{quote}

\vfill
\begin{itemize}
\item[] \hspace*{-1cm} {\em AMS Subject Class.} (2000): {\em
Primary\/} 62G05. {\em Secondary\/} 62G07. \item[] \hspace*{-1cm}
\item[] \hspace*{-1cm} {\em Short running title:} Kernel density
estimation at a parametric rate \item[] \hspace*{-1cm} Research
supported by Spanish Ministerio de Ciencia y Tecnolog\'{\i}a project
MTM2005-06348. \item[] \hspace*{-1cm} Proofs should be sent to:
Agust\'{\i}n G. Nogales, Dpto. de
    Matem\'{a}ticas, Universidad de Extremadura, Avda. de Elvas, s/n, 06071--Badajoz,
    SPAIN. e-mail: nogales@unex.es
 \end{itemize}

\newpage

\section{Introduction}
If $X_1,\dots,X_n$ is a sample from a probability distribution on
the real line with density $f$, the kernel density estimator is
given by
$$f_{n,K,h}(x)=\frac{1}{n}\sum_{i=1}^nK_h(x-X_i),$$ where the kernel
$K$ is an integrable function with $\int K=~1$, the bandwidth $h$
is a positive real number and we have used the notation
$K_h(x)=K(x/h)/h$; see, e.g., Silverman (1986), Simonoff (1996) or Wand and Jones (1995). 
The $L_2$ error criterion will be used here; that is, we will
measure the error of the estimate $f_{n,K,h}$ through the mean
integrated square error (MISE), defined by
\begin{equation*}
\mbox{MISE}_f(f_{n,K,h})=E_f\int[f_{n,K,h}(x)-f(x)]^2dx.
\end{equation*}
We will assume henceforth that all the kernels below are bounded
functions, continuous at zero and such that $\int K^2<2K(0)$. This
technical conditions ensure that an optimal bandwidth
$h_{0n}(f)=\argmin_{h>0}\text{MISE}_f(f_{n,K,h})$ exists (see Chacón
et al., 2006).

The main goal of this paper is to characterize the kernel functions
that make the MISE converge to zero as fast as possible. Most
commonly used kernels are the positive ones, because they produce
{\it bona fide} density estimators; that is, estimators that, for
every observed sample, provide a true density function (i.e.,
$f_{n,K,h}\ge 0$ and $\int f_{n,K,h}=1$). However, it is widely
known that for positive kernels the MISE cannot decrease to zero
faster than $n^{-4/5}$ (Rosenblatt, 1956). In this sense, some
benefit can be obtained if we allow the kernel to take negative
values (see Theorem 1 below), although the price to be paid is that
the resulting estimate is not a positive function. Nevertheless, in
a recent paper, Glad, Hjort and Ushakov (2003) show that, based on a
non-bona fide estimator, it is possible to construct a bona fide one
with even smaller MISE. Thus, there is no reason, in terms of MISE,
to avoid the use of kernels taking negative values in density
estimation.

Watson and Leadbetter (1963) showed, in a very general background,
that the MISE of kernel density estimators cannot decrease faster
than $n^{-1}$. Davis (1977) characterized the class of densities for
which this ``parametric" rate $n^{-1}$ can be achieved (see Theorem
3 below). In this paper, we give a characterization of those kernels
for which the MISE of the corresponding kernel estimator goes to
zero at rate $n^{-1}$ for some density, so that together with the
result of Davis (1977) we obtain a precise description of the family
of densities and kernels for which the parametric rate is attainable
(see Theorem 4). Besides, for this family we provide practical
bandwidth-choice advice for achieving this rate.

\section{Main results}

Let us recall some facts about kernels. If we denote by
$m_j(K)=\int x^jK(x)dx$ the $j$-th moment of a kernel $K$, we say
that $K$ is of finite order if the set
$$\mathcal A_K=\{j\in\N, j\ge1\colon m_j(K)\ne0\}$$
is non-empty. In this case, $k=\min\mathcal A_K$ is called the order
of the kernel $K$. If $\mathcal A_K=\varnothing$ then it is said
that $K$ is a kernel of infinite order and such a kernel should
satisfy $m_j(K)=0$ for $j=1,2,\ldots$

An example of an infinite order kernel is Natterer's kernel, whose
characteristic function is given by
$\varphi(t)=e^{-t^2/(1-t^2)}I_{[-1,1]}(t)$, where $I_A$ stands for
the indicator function of the set $A$ (see Devroye and Lugosi, 2001,
Ch. 17). If $K$ is the density of a symmetric distribution with
finite variance, then $K$ is a kernel of order 2. A method for
constructing a kernel of arbitrary finite order is shown in Schucany
and Sommers (1977); however, if we want a kernel $K$ to have order
$k>2$ then $K$ must necessarily take negative values.


Let us denote
$$\Phi(n,f,K)=\min_{h>0}\text{MISE}_f(f_{n,K,h})$$
that is, $\Phi(n,f,K)$ is the minimal MISE that can be achieved when
we use the kernel $K$ and a sample of size $n$ to estimate $f$. The
reason for using kernels of order greater than 2 (non-positive,
therefore) rests upon the following theorem, which can be found, for
instance, in Wand and Jones (1995).

\begin{teor} If $K\in L_2$ is a symmetric
kernel of finite order $k$ and the density $f$ has a $k$-th
continuous derivative belonging to $L_2$ then the minimal MISE that
may be obtained by estimating $f$ using a kernel estimator with
kernel $K$ is of exact order $n^{-2k/(2k+1)}$; that is,
$$\lim_{n\to\infty}n^{2k/(2k+1)}\Phi(n,f,K)=\alpha_1,$$
where $\alpha_1\in(0,\infty)$ is a constant depending on $f$ and
$K$.
\end{teor}

Thus, as we make the kernel order grow, the rate of convergence of
the optimal MISE to zero approaches the parametric rate $n^{-1}$,
although the class of densities for which this rate is valid gets
smaller and smaller. The question is: is there any kernel that
effectively attains the rate $n^{-1}$ for some density? The kernels
that achieve that MISE-rate for some density will deserve to be
called \emph{superkernels}; that is, a superkernel will be a kernel
$K$ which satisfies,
$$\lim_{n\to\infty} n\Phi(n,f,K)=\alpha_2$$
for some density $f$, with $0<\alpha_2<\infty$. As stated in the
previous section, our purpose here is to give a characterization of
such superkernels.


\vspace{1ex}

In view of Theorem 1, one is tempted to conjecture that an infinite
order kernel is a good candidate to be a super\-kernel; however, we
will see below that an infinite order kernel does not need to be a
superkernel.

\vspace{1ex}

Denote by $\varphi_K(t)$ 
the characteristic function of a kernel $K$ and
\begin{align*}
S_K&=\inf\{t\ge 0\colon|\varphi_K(t)-1|\neq0\}\\
T_K&=\inf\{r\ge0\colon|\varphi_K(t)-1|\neq0\text{ a.e. for }t\ge
r\}.
\end{align*}
That is, $S_K$ is the greatest value of $r$ such that $\varphi_K$ is
identically equal to 1 on $[0,r]$ and $T_K$ is the greatest value of
$t$ such that $\varphi_K(t)=1$. Notice that nearly every kernel used
in practice satisfies $S_K=T_K$.

The next result gives a characterization of the class of
superkernels, in terms of their characteristic functions.

\begin{teor}Let $K$ be a kernel in
$L_2$ such that $S_K=T_K$. The following statements are equivalent:
\begin{enumerate}
\item[i)] $S_K>0.$ \item[ii)] $\Phi(n,f,K)$ is of exact order $n^{-1}$ for some density $f\in L_2.$
\end{enumerate}
\end{teor}

The previous theorem allows us to give an alternative (and
equivalent) definition of a superkernel: we will say that a kernel
$K$ with $S_K=T_K$ is a superkernel if $S_K>0$; that is, if its
characteristic function is identically equal to 1 in a neighborhood
of the origin. This is just the classical definition of superkernel
used in Devroye (1992) or in Glad, Hjort and Ushakov (2003), for
instance. Thus, although this definition is not very intuitive,
Theorem 2 allows us to conclude that it is just the one that we were
looking for. Besides, from this characterization it follows that
Natterer's kernel, which has infinite order, is not a superkernel;
that is, the minimal MISE that we obtain using Natterer's kernel
cannot decrease to zero at rate $n^{-1}$ for any density. A
classical example of superkernel is given by the trapezoidal kernel
$K(x)=(\cos x-\cos(2x))/(\pi x^2)$, which has characteristic
function $\varphi_K(t)=I_{[0,1)}(|t|)+(2-|t|)I_{[1,2)}(|t|)$, so
that $S_K=T_K=1$; see Devroye and Lugosi (2001). Some more examples
of superkernels are included in Section 3 of McMurry and Politis
(2004), they are called infinite order flat-top kernels there.

The characterization of the class of densities for which the rate
$n^{-1}$ is attainable is given in a paper by Davis (1977). Let us
denote by $\varphi_f(t)$ the characteristic function of a density
$f$ and
\begin{align*}
C_f&=\sup\{r\ge 0\colon\varphi_f(t)\neq0\text{ a.e. for
}t\in[0,r]\}\\
D_f&=\sup\{t\ge0\colon\varphi_f(t)\neq0\}.
\end{align*}
Notice that the support of $\varphi_f$ is contained in $(-D_f,D_f)$;
moreover, this interval coincides with the support in the common
case where $C_f=D_f$.

\begin{teor}[Davis, 1977] Let $f$ be a density in $L_2$.
The following statements are equivalent:
\begin{enumerate}
\item[i)] $D_f<\infty$; i.e., $\varphi_f$ has bounded support.
\item[ii)] $\Phi(n,f,K)$ is of exact order $n^{-1}$ for some kernel
$K\in L_2.$
\end{enumerate}
\end{teor}

Davis' theorem states that in kernel density estimation the MISE may
decrease to zero at rate $n^{-1}$ only if the characteristic
function of the density we aim to estimate has bounded support. An
example of this kind of density is given by the Fejér-de la
Vallé-Poussin density, $f(x)=(1-\cos x)/(\pi x^2)$, which has
characteristic function $\varphi_f(t)=(1-|t|)I_{[-1,1]}(t)$. Davis
(1977) even provides a kernel estimator that achieves the parametric
rate if the bandwidth is properly chosen (see also Ibragimov and
Khasminskii, 1982); however, her estimator is based on the sinc
function $S(x)=(\sin x)/(\pi x^2)$, which is not a kernel as it is
not an integrable function. In contrast, our Theorem 2 is valid for
true kernel functions and gives a condition that is not only
sufficient but also necessary for kernel density estimation at a
parametric rate.

We can combine theorems 2 and 3 to get:

\begin{teor} Let $K$ be a kernel with
$S_K=T_K$ and $f$ a density, both in $L_2$. Then,
$$\Phi(n,f,K)\text{ is of exact order }n^{-1}\text{ iff }S_K>0\text{ and }D_f<\infty.$$

\end{teor}

The theorem above gives a precise characterization of the only case
where kernel density estimation at a parametric rate is possible.
Then, we may wonder what would happen if we use a superkernel when
the density does not fulfil the condition $D_f<0$, i.e., when kernel
density estimation at a parametric rate is not possible. In the
$L_1$ context, Devroye (1992) showed that superkernel estimators are
rate-adaptive, in the sense that they achieve the best possible rate
that the density permits. Below we show that this is also the case
in the $L_2$ setup.

\begin{teor}
Let $K$ be a superkernel and $f$ be a density, both in $L_2$. It is
verified:
\begin{enumerate}
\item[i)](Smooth case) If $f$ has a
$k$-th derivative in $L_1\cap L_2$, then $\Phi(n,f,K)$ goes to zero
as $n^{-2k/(2k+1)}$ or faster; that is, the sequence
$$n^{2k/(2k+1)}\Phi(n,f,K)$$
is bounded.
\item[ii)](Supersmooth case) If for some $\alpha>0$ and $\gamma>0$
the integral
$$I_{\alpha,\gamma}(f)=\int e^{\gamma|t|^\alpha}|\varphi_f(t)|^2dt$$
is finite, then $\Phi(n,f,K)$ goes to zero as $(\log
n)^{1/\alpha}/n$ or faster; that is, the sequence
$$\frac{n}{(\log n)^{1/\alpha}}\Phi(n,f,K)$$
is bounded.
\end{enumerate}
\end{teor}

\begin{rem}
We have borrowed the terminology ``smooth" and ``supersmooth" case
from Glad, Hjort and Ushakov (1999), where a result similar to our
Theorem 5 is shown for the sinc kernel; see also Davis (1977).
Notice that when $D_f<\infty$ we are in the supersmooth case for all
$\alpha>0$. Some examples of densities with
$I_{\alpha,\lambda}(f)<\infty$ include the standard Gaussian
($\alpha=2$) and Cauchy ($\alpha=1$) densities. Also, it should be
remarked that Theorem 3.1 in Politis (2003) is the analogue to the
previous result in a pointwise sense (rather than for the MISE
criterion).
\end{rem}

\begin{rem}
Denote $R(g)=\int g(x)^2dx$ for any $g\in L_2$. From the proof of
Theorem 5 (see Section 4 below), in the smooth case the quantity
$n^{2k/(2k+1)}\Phi(n,f,K)$ can be bounded by
$$(2k+1)(2k)^{-2k/(2k+1)}\left(\frac{R(K)}{S_K}\right)^{2k/(2k+1)}R(f^{(k)}).$$
For all $k$, this bound depends on the superkernel $K$ only through
$R(K)/S_K$; therefore, we could try to find the supernernel $K$
minimizing this value, as it is done in the finite-order case. For
kernels of order 2, it is well-known that the kernel minimizing an
asymptotic version of the MISE is the so-called Epanechinikov
kernel; see, e.g., Silverman (1986). Here, in the superkernel case,
we have $R(K)\geq S_K/\pi$ for all $K$. This lower bound is
achievable if and only if $\varphi_K(t)=0$ for all $|t|\geq S_K$ but
clearly, among all the superkernels satisfying such a condition, the
only one fulfilling $R(K)=S_K/\pi$ is given by
$\varphi_K(t)=I_{[-S_K,S_K]}(t)$, which corresponds to (a rescaled
version of) the sinc kernel. In this sense, although the sinc
function does not provide a proper kernel, it is the asymptotically
optimal choice; that is, the analogue to the Epanechnikov kernel for
the superkernel case.
\end{rem}

Although Theorem 4 seems to be of purely theoretical interest, as it
says nothing about the main problem in kernel density estimation,
the choice of the bandwidth, this issue may be solved by using the
next result, which can be found in Chacón et al. (2006). Let us
recall the notation $h_{0n}(f)$ for the $L_2$-optimal bandwidth;
that is,
$$h_{0n}(f)=\mathop{\text{argmin}}_{h>0}\text{MISE}_f(f_{n,K,h}).$$

\begin{teor} Let $K$ be
a kernel and $f$ a density, both in $L_2$. If $S_K=T_K$ or $C_f=D_f$
then
$$h_{0n}(f)\to S_K/D_f\text{ as }n\to\infty.$$ Moreover,
if $S_K>0$ and $D_f<\infty$ then, for any fixed
$h_\star\in(0,S_K/D_f]$ (not depending on $n$), we have
$$E_f[f_{n,K,h_\star}(x)]=f(x),\text{ for a.e. }x\in\R,\forall n\in\N,$$
so that $\mbox{MISE}_f(f_{n,K,h_\star})$ is of exact order $n^{-1}$.
\end{teor}

\begin{rem}
Theorem 6 suggests taking $h=S_K/D_f$ under the conditions of
Theorem 4. This is an asymptotic selection, as it is the limit of
the optimal bandwidth sequence but, also, in this case it provides
us with an unbiased kernel density estimator, whose MISE goes to
zero at a parametric rate. Indeed, in such a situation we can bound
$$n\Phi(n,f,K)\leq n \text{MISE}_f(f_{n,K,S_K/D_f})\leq
D_fR(K)/S_K,$$ so that same argument as in Remark 1 shows that the
sinc kernel is also the asymptotically optimal choice for the case
where  $D_f<\infty$.
\end{rem}

\begin{rem}
Any bandwidth $h_\star$ as in the previous theorem may be called a
global ``zero-bias bandwidth". In a similar way, Sain and Scott
(2002) show, for non-negative kernels, the existence of local
zero-bias bandwidths $h_0(x)$, not varying with $n$, for every $x$
in the region where $f$ is convex. Using this local bandwidths they
also get a $n^{-1}$ rate, but with respect to the pointwise mean
squared error.
\end{rem}





\section{A numerical illustration}

Next we give a simple numerical example showing the performance of
the superkernel estimators ``at full power", that is, in the optimal
situation where the characteristic function of the density has
bounded support. To do so, we are going to focus on the
aforementioned Fejér-de la Vallé-Poussin density
$$f(x)=\frac{1-\cos x}{\pi x^2},\quad x\in\R,$$
and the trapezoidal superkernel
$$K(x)=\frac{\cos x-\cos(2x)}{\pi x^2}, \quad x\in\R.$$

For this superkernel, we will use two different bandwidth selection
approaches: the first bandwidth is selected by a cross-validation
method (see Silverman, 1986, or Wand and Jones, 1995); the second
bandwidth comes from a version of the bandwidth selection procedure
proposed by Politis (2003). This method aims to estimate $D_f$
making use of the empirical characteristic function, and it is
closely related to the one proposed by Chiu (1991) for a similar
problem in density estimation (see also Politis and Romano, 1999).
If $\varphi_n(t)=n^{-1}\sum_{j=1}^n\exp\{itX_j\}$ denotes the
empirical characteristic function, $D_f$ is estimated by
$$\widehat D_n=\inf\{D>0\colon|\varphi_n(D+t)|^2<c\tfrac{\log n}{n},
\forall t\in(0,\ell_n)\},$$ where $c>0$ is a fixed constant and
$(\ell_n)$ is a positive nondecreasing sequence. As suggested in
Remark 3, the chosen bandwidth is then $\widehat h_n=1/\widehat
D_n$. Following the advice in Politis (2003), in all the simulations
we have taken $c=1$ and $\ell_n=1$.

We want to compare this superkernel density estimator with the
classical one, using a density function as a kernel. To this aim, we
also include in the simulations the results for the Sheather-Jones
method (Sheather and Jones, 1991), which uses the standard normal
density as the kernel, so that it is known that the MISE cannot
decrease faster than $n^{-4/5}$ (again, see Theorem 1 above).

We have tried these three methods for sample sizes $n=100$ (small),
$n=400$ (medium) and $n=1600$ (large) over 100 simulated samples of
each size drawn from the Fejér-de la Vallé-Poussin density. The
results are shown in Table 1. For each estimator $\hat f_n$ and
sample size we give the average and standard deviation of the 100
values of ISE$(\hat f_n)=10^3\times\int(\hat f_n-f)^2$.

\begin{table}[h]\centering
\begin{tabular}{|c|c|c|c|}\hline
n&$\text{ISE}_\text{CV}$&$\text{ISE}_\text{SJ}$&$\text{ISE}_\text{Pol}$\\\hline\hline
100&3.36&3.04&2.53\\
&(4.38)&(2.21)&(2.28)\\\hline
400&2.59&0.902&0.612\\
&(1.14)&(0.549)&(0.365)\\\hline 1600& 1.10&0.348&0.179\\
&(0.811)&(0.172)&(0.132)\\\hline
\end{tabular}
\label{tab}\caption{{\it Simulation results for sample sizes
$n=100,400,1600$. Averages and (standard deviations) of the ISE are
given for each method.}}
\end{table}

As usual, it can be seen from Table 1 that the cross-validated
selector is far more variable than the others. In this case, even
the average ISE is also unacceptably large, when it is used together
with a superkernel. In contrast, the selector of Politis does a good
work: it is comparable with the Sheather-Jones method for small
sample size, but the superior asymptotics of the superkernel
estimator clearly begin to take their advantage yet for $n=400$. For
large sample size, the better performance of the superkernel
estimator is even more evident, obtaining nearly half the average
ISE of the Sheather-Jones selector and less variance. Therefore, the
usefulness of superkernels in density estimation becomes clear, at
least in this case.

\section{Proofs}

The proof of our main result (Theorem 2) relies heavily on previous
results that may be found in Chacón et al. (2006). For the sake of
completeness we also include their statements here.

\begin{lema}
Let $f$ be a density and $K$ a kernel, both in $L_2$. It is
verified:
\begin{itemize}
\item[i) ]$R(K_h*f)\to R(f)$ as $h\to0.$
\item[ii)]If $S_K=0$ then $h_{0n}(f)\to0$ as $n\to\infty$.
\end{itemize}
\end{lema}

For the proof of Theorem 2 we will need an auxiliary result. It
states that if we use a kernel $K$ with $S_K=0$, then the
MISE-convergence rate is slower than $n^{-1}$ for every density. It
can be applied, for instance, to finite-order kernels, as it is easy
to show that any kernel of finite order satisfies $S_K=0$.

\begin{lema} If $K\in L_2$ is a kernel such
that $S_K=0$ then, for every density $f\in L_2$, we have that
$$\lim_{n\to\infty} n\Phi(n,f,K)=\infty.$$
\end{lema}

\begin{proof}
It is easy to show that
$$\textstyle\int\Var_f[f_{n,K,h}(x)]dx=R(K)/(nh)-R(K_h*f)/n,$$
where $*$ stands for convolution (see Wand and Jones, 1995).
Therefore,
\begin{align*}
n\Phi(n,f,K)&=n\mbox{MISE}_f(f_{n,K,h_{0n}(f)})\\&\geq
n\textstyle\int\Var_f[f_{n,K,h_{0n}(f)}(x)]dx
\\&=\frac{R(K)}{h_{0n}(f)}-\textstyle R(K_{h_{0n}(f)}*f)
\end{align*}
Then, the conclusion follows immediately from Lemma 1.
\end{proof}

\begin{proof}[Proof of Theorem 2]
If $S_K>0$, then Theorem 6 states that it suffices to consider a
density with $D_f<\infty$, such as the Fejér-de la Vallée-Poussin
density, to get a parametric MISE-convergence rate. On the other
hand, the previous lemma shows precisely the implication
$ii)\Rightarrow i)$.
\end{proof}

\begin{proof}[Proof of Theorem 5]
In the smooth case, standard Fourier transform theory shows that the
conditions on $f$ ensure that $$\int|t|^{2k}|\varphi_f(t)|^2dt=2\pi
R(f^{(k)})<\infty.$$ Using Parseval identity,
$2\pi\text{MISE}_f(f_{n,K,h})=B(h)+V(h),$ where
\begin{align*}
0\leq B(h)&=\int|\varphi_f(t)|^2|\varphi_K(th)-1)|^2dt\\
0\leq
V(h)&=\frac{1}{nh}\int|\varphi_K(t)|^2dt-\frac{1}{n}\int|\varphi_f(t)|^2|\varphi_K(th)|^2dt.
\end{align*}
Then, we can bound $V(h)$ by $\int|\varphi_K|^2/(nh)$ and
\begin{align*}
B(h)&=\int_{|t|>S_K/h}|\varphi_f(t)|^2|\varphi_K(th)-1|^2dt\\
&\leq\int_{|t|>S_k/h}|\varphi_f(t)|^2dt\\
&\leq\frac{h^{2k}}{S_K^{2k}}\int|t|^{2k}|\varphi_f(t)|^2dt
\end{align*}
so that
$$\text{MISE}_f(f_{n,K,h})\leq\frac{h^{2k}}{S_K^{2k}}R(f^{(k)})+\frac{R(K)}{nh}.$$
Calculating the minimum of the expression on the right-hand-side of
the previous display, we get
$$\Phi(n,f,K)\leq Cn^{-2k/(2k+1)},$$
as desired.

For the supersmooth case, the same kind of calculations can be used
to bound
$$B(h)\leq e^{-S_K\gamma/h^\alpha}I_{\alpha,\gamma}(f).$$
Now, taking $h$ to be of order $(\log n)^{-1/\alpha}$ in $B(h)+V(h)$
gives the proof.
\end{proof}

\end{document}